\documentclass[12pt,reqno]{amsart}
\usepackage{amssymb,epsf}
\usepackage{amstext,amsmath,amsfonts,amscd,amsthm,graphicx}
\usepackage{upref}
\usepackage{epsfig}
\usepackage[utf8]{inputenc}
\usepackage{latexsym}
\usepackage{eucal}
\usepackage{wrapfig}
\newfont{\fnt}{cmsy10}
\newfont{\sss}{cmti10}

\theoremstyle{definition}
\newtheorem{df}{Definition}[section]
\theoremstyle{plain}
\newtheorem{thm}{Theorem}[section]
\newtheorem{prp}{Proposition}[section]
\theoremstyle{definition}
\newtheorem{rmk}{Remark}[section]

\begin{document}
\author{Barbora Volná}
\address{Mathematical Institute, Silesian University in Opava, \newline \indent Na Rybníčku 1, 746 01 Opava, Czech Republic}
\email{Barbora.Volna@math.slu.cz}
\keywords{IS-LM model, Relaxation Oscillations, Macroeconomic Stability}
\subjclass[2010]{91B50, 91B55}

\title{Relaxation Oscillations in New IS-LM Model}

\begin{abstract}
In this paper, we create new version of IS-LM model. The original IS-LM model has two main deficiencies: assumptions of constant price level and of strictly exogenous money supply. New IS-LM model eliminates these deficiencies.\\
\indent
In the second section, we prove the existence of relaxation oscillations in this new IS-LM model.
\end{abstract}

\maketitle

\section*{Introduction}
In this paper, we focused on IS-LM model which is one of the fundamental macroeconomic models explaining aggregate macroeconomic equilibrium, i. e. goods market equilibrium (represented by the curve IS) and money market or financial assets market equilibrium (represented by the curve LM) simultaneously. Many experts researched the problem of macroeconomic (un)stability and also the IS-LM model. Already in 1937, J. \nolinebreak R. \nolinebreak Hicks published the original IS-LM model, see \cite{hicks}, which is mathematical model of Mr. Keynes theory. Many versions of this model and related problems were presented in several works, see e. g. \cite{barakova, cesare-sportelli, chiba-leong, gandolfo, king, kodera, varian, zhou-li, zhou-li_2}. \\
\indent
So, the IS-LM model takes unquestionably some place in the macroeconomics area. The \nolinebreak original IS-LM model has many supporters and many opposers. This model is mainly criticized by opposers for missing modelling of inflation effect and for his conception of supply of money. So, the original IS-LM model has these two main deficiencies. The first one means an assumption of constant price level and the second one means an assumption of strictly exogenous money supply, i. e. supply of money is certain constant money stock determined by central bank. Today's economists can not find consensus in the problem of endogenity or exogenity of money supply, see e. g. \cite{izak}. Endogenous money supply means that money are generated in economics by credit creation, see e. g. \cite{sojka}. We created new version of IS-LM model which eliminates these deficiencies. We include modelling inflation and also some money supply function, which represents conjunction of endogenous and exogenous conception of money supply, to this model. \\
\indent
So, there is brief presentation of original IS-LM model with some standard economic properties of the IS-LM model functions and then there is definition of the new version of IS-LM model with another properties counting for example the sufficient condition for existence of aggregate macroeconomic equilibrium (the intersection point of the curve IS and the curve LM).\\
\indent
The second section of this paper aims at relaxation oscillations in this new IS-LM model. Many experts were modelling different type of oscillations in various branch of science (physics, ecology, biology, economy) for several centuries, see e. g. \cite{bogoliubov-mitropolsky, ding, guckenheimer-holmes, heer_maussner, kaldor, kolesov-kolesov, perko, puu, surhone_timpledon_marseken, zhang}. Today's situation and unexpected fluctuations in the economics confirm the necessity of modelling different types of oscillations occurring in the economics. We use new IS-LM model and show how can this special type of oscillation resembled limit cycle called relaxation oscillation can originate in this model. The relaxation oscillations are the fluctuations of the main variables caused by different velocity of the trajectories and by the (un)stability of the arcs of the relevant isoclines. Relaxation oscillations may explain the situation, when one unstable regime is located between two stable regimes.\\
\indent
We formulate the sufficient conditions for existence of relaxation oscillations in this new IS-LM model. These conditions are not only mathematical conditions, but also there exists the relevant economical theory of these conditions and relevant interpretation what these conditions mean in economics. There can originate the relaxation oscillations on IS side of the model, i. e. on goods market side, and also on LM side of the model, i. e. on money market or financial assets market side. For the modelling of relaxation oscillations on IS side, we presume that the subjects and their reactions on goods market are faster than on money market (or financial assets market). Vice versa for the modelling of relaxation oscillations on LM side, we suppose that the subjects and their reactions on money market (or financial assets market) are faster than on goods market. The relevant economic theory of the first conditions for IS side is Nicolas Kaldor theory about special form of investment and saving function and the theory of the second conditions for LM \nolinebreak side is some little fluctuation of money demand and supply behaviour called liquidity trap. We describe the dynamical behaviour of these systems and we prove the existence of relaxation oscillations and the stability and unstability of the arcs of relevant isoclines. This quick "jumps" emerging in the relaxation oscillations can ascribe unusual behaviour of economics and new strange phenomenons occurring in economics. \\
\indent
The cycle, which originates in the new IS-LM model, is caused by the shifting of the curve IS or LM. This shifting can be explained like the demonstration of the fiscal or monetary policy. So, this new IS-LM model with relaxation oscillations can be first approximation of the estimation of the government intervention impacts.

\section{Basic Notations and Main Economic Functions}

\subsection{Basic Notations \\}

\begin{tabular}{lp{10cm}}
  $Y$             & aggregate income (GDP, GNP), \\
  $R$             & long-term real interest rate, \\
  $i_S$           & short-term nominal interest rate, \\
  $I$             & investments, \\
  $S$             & savings, \\
  $L$             & demand for money, \\
  $M$			  & supply of money, \\
  $\alpha, \beta$ & parameters of dynamic.
\end{tabular}

\subsection{Main Economic Functions of the IS-LM model \\} 

\begin{tabular}{lp{10cm}}
  $I(Y,R)$           & investment function,\\
  $S(Y,R)$           & saving function, \\
  $L(Y,R), L(Y,i_S)$ & demand for money function,\\
  $M_S$              & constant representing money supply, \\
  $M(Y,i_S)$	     & money supply function.
\end{tabular} \\
\\
\indent
We suppose that all of these functions are continuous and differentiable.

\section{IS-LM Models}

\subsection{Original IS-LM Model}

In this subsection, we briefly present original IS-LM model, its economic assumptions and some standard economic properties of the main economic functions, see e. g. \cite{gandolfo}. \\
\indent
The \textit{economic assumptions of the original IS-LM model} are following:
\begin{itemize}
\item a two-sector economy (households - firms),
\item a demand-oriented model (the supply is fully adapted to demand),
\item $Y \geq 0,~R>0 $,
\item a constant price level (i. e. absence of inflation effect),
\item a money supply is an exogenous quantity (means that money supply is some money stock determined by Central Bank).
\end{itemize}

\begin{df}
The \textit{original IS-LM model} is given by the following system of two algebraic equations 
\begin{equation}
\label{original_static_IS-LM}
\begin{array}{ll}
\textrm{IS:} & I(Y,R)=S(Y,R) \\
\textrm{LM:} & L(Y,R)=M_S,
\end{array}
\end{equation}
where $M_S > 0$, in the static form and by this system of two differential equations
\begin{equation}
\label{original_dynamic_IS-LM}
\begin{array}{ll}
\textrm{IS:} & \frac{d Y}{d t} = \alpha [I(Y,R)-S(Y,R)] \\
\textrm{LM:} & \frac{d R}{d t} = \beta [L(Y,R)-M_S],
\end{array}
\end{equation}
where $\alpha, \beta >0 $, in the dynamic form.
\end{df}

Economic theory puts on the main economic functions of this model some properties. We can present these properties using the following formulas:
\begin{equation}
\label{economic_I} 
  0<\frac{\partial I}{\partial Y}<1, \frac{\partial I}{\partial R}<0,
\end{equation}
\begin{equation}
\label{economic_S}
  0<\frac{\partial S}{\partial Y}<1, \frac{\partial S}{\partial R}>0,
\end{equation}
\begin{equation}
\label{economic_L}
  \frac{\partial L}{\partial Y}>0, \frac{\partial L}{\partial R}<0. 
\end{equation}

The one of the main outcomes of IS-LM model is intersection point of the curve IS and LM representing aggregate macroeconomic equilibrium (i. e. the equilibrium on goods market and on money market, or on financial assets market, simultaneously). This original IS-LM model in general form need not to give us some equilibrium (if there is no intersection point of these curves). \\
\indent
Standard course of the curve IS is decreasing (in general) and of the curve LM is increasing (in general) due to economic sense. It is possible that these curves are non-standard. Either this non-standard course exists only for short part of the curve and describes some different or unusual behaviour for this part or this non-standard course exists for whole curve and describes unusual behaviour for ever. First possibility is more common and second possibility is rather rare.

\subsection{New IS-LM model}

In this subsection, we propose new IS-LM model which tries to eliminate two main deficiencies of original IS-LM model (see last two points in economic assumptions of this model). These deficiencies are the presumption of constant price level and the presumption of strictly exogenous money supply.\\
\indent
The \textit{economic assumptions of the new IS-LM model} are following:
\begin{itemize}
\item a two-sector economy (households - firms),
\item a demand-oriented model (the supply is fully adapted to demand),
\item $Y \geq 0,~R \in \mathbb{R},~i_S \in \mathbb{R}^+$,
\item a variable price level (i. e. included inflation effect),
\item a conjunction of endogenous and exogenous money supply (means that money are generated by credit creation with some certain intervention of Central Bank).
\end{itemize}

We add to original IS-LM model a floating price level, i. e. inflation effect. So, we need to distinguish two type of interest rate - a long-term real interest rate and a short-term nominal interest rate in addition to original IS-LM model. There is the long-term real interest rate on goods market and the short-term nominal interest rate on money market (or financial assets market). It holds well-known relation
\begin{equation}
i_S = R - MP + \pi^e, 
\end{equation}
where $MP$ is a maturity premium and $\pi^e$ is an inflation rate. While $MP$ and $\pi^e$ are constants, it holds $\frac{d i_S}{d t} = \frac{d (R - MP + \pi^e)}{d t} = \frac{d R}{d t}$.\\
\indent
Then, we consider that the money supply is not strictly exogenous quantity, but the supply of money is endogenous quantity (money are generated in economics by credit creation) with some exogenous part (certain money stock determined by Central bank).

\begin{df}
\label{definition_M}
We define the \emph{supply of money} by the formula
\begin{equation}
  M(Y,i_S) + M_S  ,
\end{equation}
where $M(Y,i_S)$ represents the endogenous part of money supply and $M_S>0$ represents the exogenous part of money supply. 
\end{df}

So, there are the functions $L(Y,i_S)=L(Y,R - MP + \pi^e)$ and $M(Y,i_S)=M(Y,R - MP + \pi^e)$ on the money market. It also holds $\frac{\partial L(Y,i_S)}{\partial i_S}=\frac{\partial L(Y,R - MP + \pi^e)}{\partial R}$ and $\frac{\partial M(Y,i_S)}{\partial i_S}=\frac{\partial M(Y,R - MP + \pi^e)}{\partial R}$ because of constant $MP$ and $\pi^e$.\\
\indent
Now, we can define the new IS-LM model.

\begin{df}
The \textit{new IS-LM model} is given by the following system of two algebraic equations 
\begin{equation}
\label{new_static_IS-LM}
\begin{array}{ll}
\textrm{IS:} & I(Y,R)=S(Y,R) \\
\textrm{LM:} & L(Y,R- MP + \pi^e)=M(Y,R - MP + \pi^e)+M_S,
\end{array}
\end{equation}
where $M_S > 0$, in the static form and by this system of two differential equations
\begin{equation}
\label{new_dynamic_IS-LM}
\begin{array}{ll}
\textrm{IS:} & \frac{d Y}{d t} = \alpha [I(Y,R)-S(Y,R)] \\
\textrm{LM:} & \frac{d R}{d t} = \beta [L(Y,R- MP + \pi^e)-M(Y,R - MP + \pi^e)-M_S],
\end{array}
\end{equation}
where $\alpha, \beta >0 $, in the dynamic form.
\end{df}

The investment, saving and money demand function have the same standard economic properties (\ref{economic_I}), (\ref{economic_S}) and (\ref{economic_L}). We have to put on our new function of money supply the standard economic properties. These properties are
\begin{equation}
\label{economic_M_Y} 
0 < \frac{\partial M}{\partial Y} < \frac{\partial L}{\partial Y},
\end{equation}
\begin{equation}
\label{economic_M_R} 
\frac{\partial M}{\partial R} > 0.
\end{equation}
The formula (\ref{economic_M_Y}) means that the relation between supply of money and aggregate income is positive and that the rate of increase of money supply depending on aggregate income is smaller than the rate of increase of money demand depending on aggregate income because the banks are more cautious than another subjects. And the formula (\ref{economic_M_R}) means that the relation between supply of money and interest rate is positive.\\
\indent
It is easy to see (using the Implicit Function Theorem) that the course of the curve LM is increasing in this new IS-LM model with properties (\ref{economic_L}), (\ref{economic_M_Y}) and (\ref{economic_M_R}). If it also holds
\begin{equation}
\label{economic_I_S}
\frac{\partial I}{\partial Y} < \frac{\partial S}{\partial Y},
\end{equation}
in additions to the properties (\ref{economic_I}) and (\ref{economic_S}), then the course of the curve IS will be decreasing. Now, we denote the function, whose graph is the curve IS, as $R_{IS}(Y)$ and the function, whose graph is the curve LM, as $R_{LM}(Y)$. These functions exist because of the Implicit Function Theorem. Now, if we assume
\begin{equation}
\label{condition_for_intersection_point}
\lim_{Y \rightarrow 0^+} R_{IS}(Y) > \lim_{Y \rightarrow 0^+} R_{LM}(Y),
\end{equation}
then there will exist at least one intersection point of the curve IS and LM.\\
\indent
These presented model is IS-LM model without two main deficiencies of the original model and moreover we are certain that there exists at least one aggregate macroeconomic equilibrium.

\section{New IS-LM Model and Relaxation Oscillations}

The relaxation oscillations can originate in the model described by the system of two differential equations when we suppose that one of main variable are changing very slowly in time in proportion to the second variable. So, the derivation of this variable by time is approaching to zero. In our model, the main variables are the aggregate income $Y$ and the interest rate $R$. So, there can exist relaxation oscillations on the goods market and on the money market (or financial assets market).

\subsection{Relaxation Oscillations on Goods Market}

In this case, we assume that the subjects and their reactions on goods market are faster than on money market. So, we suppose that the interest rate $R$ is changing very slowly in time in proportion to the aggregate income $Y$. We can describe this  situation by following equations
\begin{equation}
\label{dynamic_model_IS_parameter} 
  \begin{array}{lll}
  \frac{d Y}{d t} & = & \alpha [I(Y,R)-S(Y,R)]  \\
  \frac{d R}{d t} & = & \epsilon \beta [L(Y,R - MP + \pi^e)-M(Y,R - MP + \pi^e)-M_S]  
  \end{array}
\end{equation}
where $\epsilon \rightarrow 0$ is some very small positive parameter. \\
\indent
If this parameter $\epsilon$ is very small, then we can consider $\frac{d R}{d t} = 0$ and we can write previous equations in the following forms.
\begin{equation}
\label{dynamic_model_IS_parameter_2} 
  \begin{array}{lll}
  \frac{d Y}{d t} & = &  \alpha [I(Y,R)-S(Y,R)]  \\
  \frac{d R}{d t} & = & 0  
  \end{array}
\end{equation}

Then, there remains only the curve IS, called isocline of the equation IS. The curve LM does not exist here considering $\frac{d R}{d t} = 0$. Every points of the curve IS are singular points.\\
\indent
Now, we need to formulate new conditions put on the investment and saving function. These conditions will be sufficient for existence of relaxation oscillations on goods market (it is described below). We tend to Nicolas Kaldor's theory about investment and saving function. He described how these functions should look like already in 1940, see \cite{kaldor}. The investment and saving function depends only on $Y$ for some fixed $R$ ($I(Y)$ and $S(Y)$) has so-called "sigma-shaped" graphs, see Figure \ref{fig:I_and_S_Kaldor}.

\begin{figure}[ht]
  \centering
  \includegraphics[height=4cm]{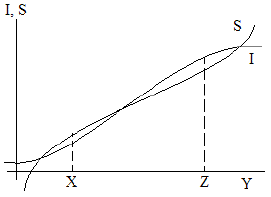}
  \caption{Shapes of the graphs of $I(Y)$ and $S(Y)$}
  \label{fig:I_and_S_Kaldor}
\end{figure}

Now, we use (with another designations) the description of these properties from Baráková, see \cite{barakova}. So-called \textit{"Kaldor's" conditions} are
\begin{equation}
\label{Kaldor_IS}
  \begin{array}{lll}
  \frac{\partial I}{\partial Y} < \frac{\partial S}{\partial Y} & \textrm{for} & Y \in [0,X), \\
  \frac{\partial I}{\partial Y} > \frac{\partial S}{\partial Y} & \textrm{for} & Y \in (X,Z), \\
  \frac{\partial I}{\partial Y} < \frac{\partial S}{\partial Y} & \textrm{for} & Y \in (Z,\infty), \\
  \end{array}
\end{equation}
where points $X < Z$ are given by equation $\frac{\partial I}{\partial Y} = \frac{\partial S}{\partial Y}$ for some fixed $R$.\\
\indent
It is easy to see (using the Implicit Function Theorem) that isoclines (curve IS and curve LM) of the system (\ref{dynamic_model_IS_parameter}) without requirement of $\epsilon \rightarrow 0$ and with properties (\ref{economic_I}), (\ref{economic_S}), (\ref{economic_L}), (\ref{economic_M_Y}), (\ref{economic_M_R}) and (\ref{Kaldor_IS}) are the following.

\begin{figure}[ht]
  \centering
  \includegraphics[height=5cm]{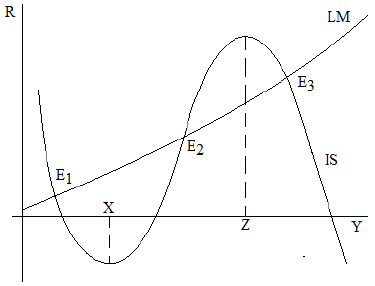}
  \caption{Curves IS and LM in the model (\ref{dynamic_model_IS_parameter}) (excluding $\epsilon \rightarrow 0$)}
  \label{fig:curve_IS_LM_1}
\end{figure}

There can occur at least one and the most three singular points. We can see this using the shifting of the curve IS or LM downwards or upwards.\\
\indent
Now, we consider the system (\ref{dynamic_model_IS_parameter}) with requirement of $\epsilon \rightarrow 0$, so we can consider the system (\ref{dynamic_model_IS_parameter_2}) in the simplified way. There we can see that the variable $R$ is a parameter in equations $\frac{d Y}{d t} =  \alpha [I(Y,R)-S(Y,R)]$. The curve IS determined by the equation $\alpha [I(Y,R)-S(Y,R)]=0$ in this situation (where $R$ is considered as a parameter in this equation), is displayed in the following Figure \ref{fig:curve_IS}.

\begin{figure}[ht]
  \centering
  \includegraphics[height=5cm]{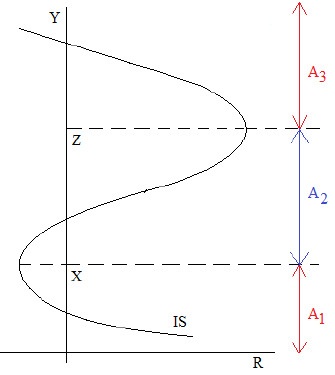}
  \caption{Curve IS considering $R$ as a parameter}
  \label{fig:curve_IS}
\end{figure}

There, every points of the curve IS are singular points.\\
\newpage
\indent
We need to distinguish the stable and unstable arcs of the isocline IS.

\begin{prp} (On stable and unstable arcs of the curve IS)
\label{prp:stable_unstable_arcs_1}
\begin{enumerate}
\item The arcs $A_1$ and $A_3$ of the curve IS are stable arcs.
\item The arc $A_2$ of the curve IS is unstable arc.
\end{enumerate}
\end{prp}
\begin{proof}
If we consider the system (\ref{dynamic_model_IS_parameter_2}), where variable $R$ is considered as a parameter in the equation $\frac{d Y}{d t} =  \alpha [I(Y,R)-S(Y,R)]$, then the stability and unstability of these arcs will be given by qualification of possible singular points in the system (\ref{dynamic_model_IS_parameter}) excluding $\epsilon \rightarrow 0$ here and below in this proof. \\
\indent
If we require (\ref{condition_for_intersection_point}) (and we consider only $Y \geq 0$), then there exist at least one and at the most three singular points (see Figure \ref{fig:curve_IS_LM_1} and try to shift the curve IS or LM downwards or upwards). Let denote $J$ as the Jacobi's matrix of the system (\ref{dynamic_model_IS_parameter}) and $\lambda_{1,2}$ as the eigenvalues of this Jacobi's matrix. It holds
\begin{equation}
\lambda_{1,2} = \frac{1}{2} \left[ \alpha (I_Y - S_Y) + \epsilon \beta (L_R - M_R) \pm \sqrt{\left[ \alpha (I_Y - S_Y) + \epsilon \beta (L_R - M_R)  \right]^2 - 4detJ} \right] \nonumber
\end{equation}
where $det J = \alpha \epsilon \beta \left[ (I_Y - S_Y)(L_R-M_R) - (I_R - S_R) (L_Y - M_Y) \right]$.\\
\indent
Now, we consider three singular points. The singular point $E_1$ or $E_3$ lies in the part of arc $A_1$ or $A_3$ (see the Figure \ref{fig:curve_IS_LM_1} and \ref{fig:curve_IS}). So, we have to prove that these two singular points are stable points. Then, this implies that arcs $A_1$ and $A_3$ are stable arcs. The real part of eigenvalues $Re(\lambda_{1,2}) < 0$ in the point $E_1$ or $E_3$ because it holds $I_Y < S_Y$ according to Kaldor's condition (\ref{Kaldor_IS}) (the point $E_1$ or $E_3$ lies in the area of $Y \in [0,X] \cup [Z, \infty)$), $L_R - M_R<0$ according to (\ref{economic_L}) and (\ref{economic_M_R}) and it holds $det J>0$ according to economic condition (\ref{economic_I}), (\ref{economic_S}), (\ref{economic_L}), (\ref{economic_M_Y}), (\ref{economic_M_R}) and Kaldor's conditions (\ref{Kaldor_IS}). From this follows that $E_1$ and $E_3$ are attractors. \\
The singular point $E_2$ lies in the part of arc $A_2$ (see the Figure \ref{fig:curve_IS_LM_1} and \ref{fig:curve_IS}). So, we have to prove that this singular point are unstable point, more precisely unstable saddle point. Then, this implies that arcs $A_2$ is unstable arc. We can see that the slope of the curve IS is greater than the slope of the curve LM in the point $E_2$ (see Figure \ref{fig:curve_IS_LM_1}). Using the Implicit Function Theorem it holds $-\frac{I_Y - S_Y}{I_R - S_R}>-\frac{L_Y - M_Y}{L_R - M_R}$. This and the conditions (\ref{economic_I}), (\ref{economic_S}), (\ref{economic_L}), (\ref{economic_M_Y}), (\ref{economic_M_R}) and Kaldor's conditions (\ref{Kaldor_IS}) (the point $E_2$ lies in the area of $Y \in (X,Z)$) imply that determinant of Jacobi's matrix of the system (\ref{dynamic_model_IS_parameter}) in the point $E_2$ is negative. From this follows that $E_2$ is unstable saddle point. \\
\indent
Now, we consider two singular points. First singular point is qualitatively the same as the point $E_1$ or $E_3$ from previous situation with three singular point. This singular point lies in the area of the arc $A_1$ or $A_3$. So, these arcs are stable. \\
The second point describe levels of $Y$ and $R$ where the curve IS and the curve LM has the same slope. Using the Implicit Function Theorem it holds $-\frac{I_Y - S_Y}{I_R - S_R}=-\frac{L_Y - M_Y}{L_R - M_R}$. This implies zero determinant of Jacobi's matrix of the system (\ref{dynamic_model_IS_parameter}) in this point. This leads to at least one zero eigenvalue. This singular point is surely unstable and lies in the area of the arc $A_2$. From this follows that $A_2$ is unstable also in this situation of two singular points. \\
\indent
If there originates only one singular point, this singular point will be qualitatively the same as the point $E_1$ or $E_3$ from situation with three singular point. This singular point lies in the area of the arc $A_1$ or $A_3$. So, these arcs are stable.
\end{proof}

\begin{thm} (On existence of relaxation oscillations on goods market) \\
Let us consider the new IS-LM model with very slow changes of variable $R$ in time (\ref{dynamic_model_IS_parameter}) (including $\epsilon \rightarrow 0$) with economic properties (\ref{economic_I}), (\ref{economic_S}), (\ref{economic_L}), (\ref{economic_M_Y}), (\ref{economic_M_R}) and Kaldor's conditions (\ref{Kaldor_IS}). Then there exist the clockwise relaxation oscillations.
\end{thm}

\begin{figure}[ht]
  \centering
  \includegraphics[height=6.5cm]{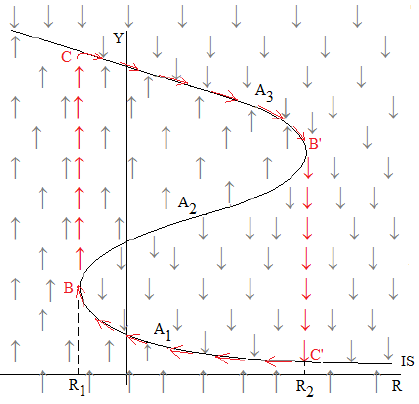}
  \caption{Relaxation oscillations in the model (\ref{dynamic_model_IS_parameter})}
  \label{fig:relaxation_oscillations_IS}
\end{figure}

\begin{proof}
If we have new IS-LM model with very slow changes of variable $R$ in time (\ref{dynamic_model_IS_parameter}) considering $\epsilon \rightarrow 0$, then we can consider the system (\ref{dynamic_model_IS_parameter_2}) instead the system (\ref{dynamic_model_IS_parameter}).\\
\indent
Kaldor's conditions (\ref{Kaldor_IS}) and the economic properties (\ref{economic_I}), (\ref{economic_S}) of the investment and saving function ensure the form of the curve IS using the Implicit Function Theorem, see Figure \ref{fig:curve_IS}.\\
\indent
The trajectories of the system (\ref{dynamic_model_IS_parameter_2}) are directed almost vertically downwards or upwards (parallel to axis $Y$) considering $\frac{d R}{d t} = 0$, see the Figure \ref{fig:relaxation_oscillations_IS}. Up or down direction of the trajectories is given by the sign of the function $\beta [L(Y,R - MP + \pi^e)-M(Y,R - MP + \pi^e)-M_S]$ on the curve IS. So, the direction of these trajectories is dependent on the stability or unstability of the arcs $A_1$, $A_2$ and $A_3$, see the Proposition \ref{prp:stable_unstable_arcs_1}. The trajectories are attracted to the stable arcs $A_1$ or $A_3$ and are drove away the unstable arc $A_2$. The speed of trajectories are finite near the isocline IS and nearness of the curve IS the trajectories go along the curve IS. The speed of trajectories are infinite large elsewhere.\\
\indent
For $R<R_1$, $R>R_2$ there exists only one stable stationary state and for $R_1<R<R_2$ there exist three stationary states: two stable and one unstable between them. Every points of the curve IS are stationary points. $R_1, R_2$ are bifurcations values. If the parameter  $R$ is changing very slowly, then the coordinates of stationary points are changed as displayed in the Figure \ref{fig:relaxation_oscillations_IS}.\\
\indent 
Now, we construct the cycle which is one vibration of the relaxation oscillations. We are changing the parameter $R$ from the level $R_2$ to $R_1$. If the moving point is on or near the stable arc $A_1$, the moving point will go along this stable arc $A_1$, then it will pass the unstable arc $A_2$ from point $B$ to $C$, see Figure \ref{fig:relaxation_oscillations_IS}. The velocity between the point $B$ and $C$ is infinite large. There originates some "jump". There is the similar situation if we are changing the parameter $R$ from the level $R_1$ to $R_2$. If the moving point is on or near the stable arc $A_3$, the moving point will go along the stable arc $A_3$ and then the moving point is attracted from point $B'$ to $C'$, see Figure \ref{fig:relaxation_oscillations_IS}. The sign of the function $\beta [L(Y,R - MP + \pi^e)-M(Y,R - MP + \pi^e)-M_S]$ may change along the curve IS in such a way, that these trajectories form some path resembled limit cycle called relaxation oscillation. This oscillation contains the trajectories described by stable arcs $A_1$ and $A_3$ with finite velocity and the trajectories described by vertical segments (between points $B$ and $C$ and also between points $B'$ and $C'$) with infinite velocity (looking like a "jump"), see the Figure \ref{fig:relaxation_oscillations_IS}. These trajectories form clockwise cycle because of the form of the curve IS and stability or unstability of the relevant arcs. Thus, we have just constructed the clockwise cycle which is one vibration of the clockwise relaxation oscillations.
\end{proof}

We can display this oscillation using the isoclines IS and LM with shifting curve LM downwards and upwards, see the Figure \ref{fig:relaxation_oscillations_IS-LM_1}.

\begin{figure}[ht]
  \centering
  \includegraphics[height=5cm]{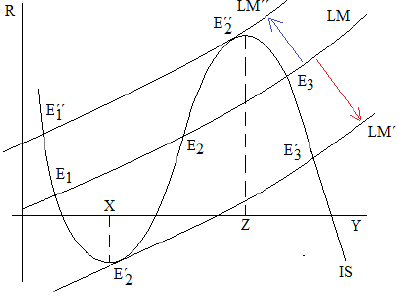}
  \caption{Shifts of the curve LM}
  \label{fig:relaxation_oscillations_IS-LM_1}
\end{figure}

In this figure, we can see three singular points $E_1$, $E_2$ and $E_3$ (intersection points of IS and LM). $E_1$ and $E_3$ are stable singular points, $E_2$ is unstable saddle point. Then the curve LM shifts upwards or downwards according the changing of the parameter $R$ until there are two singular points $E_1''$ and $E_2''$ (intersection points of IS and LM''), or $E_2'$ and $E_3'$ (intersection points of IS and LM'). The point $E_1''$ and $E_3'$ is stable node or stable focus and the point $E_2'$ and $E_2''$ is unstable, see the proof of the Proposition \ref{prp:stable_unstable_arcs_1}.

\subsection{Relaxation Oscillations on Money Market, or Financial Assets Market}

In this case, we assume that the subjects and their reactions on money market are faster than on goods market. So, we suppose that aggregate income $Y$ is changing very slowly in time in proportion to the interest rate $R$. We can describe this  situation by following equations
\begin{equation}
\label{dynamic_model_LM_parameter} 
  \begin{array}{lll}
  \frac{d Y}{d t} & = & \epsilon \alpha [I(Y,R)-S(Y,R)]  \\
  \frac{d R}{d t} & = & \beta [L(Y,R - MP + \pi^e)-M(Y,R - MP + \pi^e)-M_S]  
  \end{array}
\end{equation}
where $\epsilon \rightarrow 0$ is some very small positive parameter.\\
\indent
If this parameter $\epsilon$ is very small, then we can consider $\frac{d Y}{d t} = 0$ and we can write previous equations in the following forms.
\begin{equation}
\label{dynamic_model_LM_parameter_2} 
  \begin{array}{lll}
  \frac{d Y}{d t} & = & 0  \\
  \frac{d R}{d t} & = & \beta [L(Y,R - MP + \pi^e)-M(Y,R - MP + \pi^e)-M_S]
  \end{array}
\end{equation}

Then, there remains only the curve LM, called isocline of the equation LM. The curve IS does not exist here considering $\frac{d Y}{d t} = 0$. Every points of the curve LM are singular points.\\
\indent
Now, we need to formulate new conditions put on the demand for money and supply of money function. These conditions will be sufficient for existence of relaxation oscillations on money or financial assets market (it is described below). We suppose some unusual behaviour of the demand for money and supply of money. We assume three phases of the course of money demand and money supply function depending on $i_S$ for some fixed $Y$. In the first phase, for $i_S \in [0, P), P>0$, these subjects on the money (or financial assets) market behave usual and the course of the money demand and money supply function is standard how we describe in the condition (\ref{economic_L}) and (\ref{economic_M_R}). In the second phase, for $i_S \in (P, Q), P < Q$, these subjects behave unusual, precisely reversely. We can describe this behaviour using following formula
\begin{equation}
\label{unusual_economic_L_R} 
\frac{\partial L}{\partial i_S} = \frac{\partial L}{\partial R} > 0,
\end{equation}
\begin{equation}
\label{unusual_economic_M_R} 
\frac{\partial M}{\partial i_S} = \frac{\partial M}{\partial R} < 0.
\end{equation}
These properties correspond to unusual economic situation called liquidity trap. This means that the subjects on money or financial assets market (demand side) prefer liquidity despite relatively high level of (short-term nominal) interest rate. They own money rather than stocks, although they could have bigger gain because of relative high level of (short-term nominal) interest rate. This "irrational" behaviour of these subjects can be caused by big risk of holding these stocks and by small willingness to undergo this risk. The supply of money fully adapts to money demand (we assume demand-oriented model). This phase should be small. In the third phase, for $i_S \in (Q, \infty)$, these subjects behave usual as in the first phase. We can see the graphs of money demand and money supply function describing this behaviour on the following Figure \ref{fig:L(i_S)_and_M(i_S)_3_phases}.

\begin{figure}[ht]
  \centering
  \includegraphics[height=5cm]{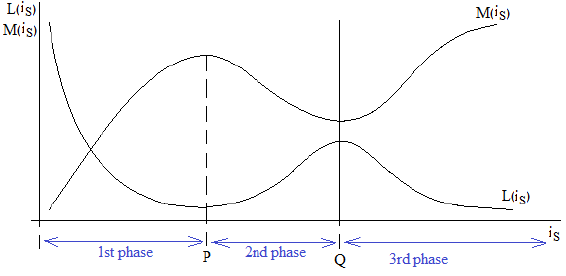}
  \caption{Three phases of the graphs of $L(i_S)$ and $M(i_S)$}
  \label{fig:L(i_S)_and_M(i_S)_3_phases}
\end{figure}

\begin{rmk}
$\frac{\partial L}{\partial i_S} = \frac{\partial L}{\partial R} = 0$ and $\frac{\partial M}{\partial i_S} = \frac{\partial M}{\partial R} = 0$ in the point $P$ and $Q$. 
\end{rmk}

In the text below, we will call this behaviour of money demand and money supply like \textit{three phases money demand and money supply}.\\
\indent
It is easy to see (using the Implicit Function Theorem and we can change axis) that isoclines (curve IS and curve LM) of the system (\ref{dynamic_model_LM_parameter}) without requirement of $\epsilon \rightarrow 0$ and with properties (\ref{economic_I}), (\ref{economic_S}), (\ref{economic_I_S}) and three phases money demand and money supply are the following.

\begin{figure}[ht]
  \centering
  \includegraphics[height=5cm]{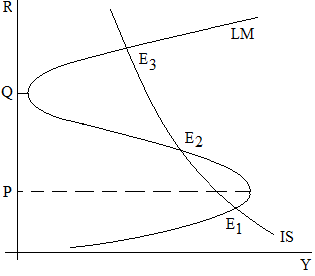}
  \caption{Curves IS and LM in the model (\ref{dynamic_model_LM_parameter}) (excluding $\epsilon \rightarrow 0$)}
  \label{fig:curve_IS_LM_2}
\end{figure}

There can occur at least one and the most three singular points. We can see this using the shifting of the curve IS or LM downwards or upwards.\\
\indent
Now, we consider the system (\ref{dynamic_model_LM_parameter}) with requirement of $\epsilon \rightarrow 0$, so we can consider the system (\ref{dynamic_model_LM_parameter_2}) in the simplified way. There we can see that the variable $Y$ is a parameter in equations $\frac{d R}{d t} =  \beta [L(Y,R - MP + \pi^e)-M(Y,R - MP + \pi^e)-M_S]$. The curve LM determined by the equation $\beta [L(Y,R - MP + \pi^e)-M(Y,R - MP + \pi^e)-M_S]=0$ in this situation (where $R$ is considered as a parameter) is displayed in the following Figure \nolinebreak \ref{fig:curve_LM}.

\begin{figure}[ht]
  \centering
  \includegraphics[height=5cm]{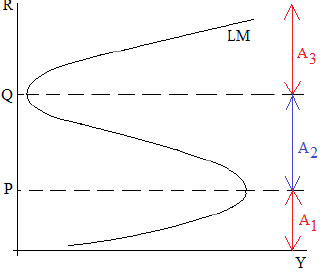}
  \caption{Curve LM considering $Y$ as a parameter}
  \label{fig:curve_LM}
\end{figure}

There, every points of the curve IS are singular points.\\
\newpage
\indent
We need to distinguish the stable and unstable arcs of the isocline LM.

\begin{prp} (On stable and unstable arcs of the curve LM)
\label{prp:stable_unstable_arcs_2}
\begin{enumerate}
\item The arcs $A_1$ and $A_3$ of the curve LM are stable arcs.
\item The arc $A_2$ of the curve LM is unstable arc.
\end{enumerate}
\end{prp}
\begin{proof}
If we consider the system (\ref{dynamic_model_LM_parameter_2}), where variable $Y$ is considered as a parameter in the equation $\frac{d R}{d t} =  \beta [L(Y,R - MP + \pi^e)-M(Y,R - MP + \pi^e)-M_S]$, then the stability and unstability of these arcs will be given by qualification of possible singular points in the system (\ref{dynamic_model_LM_parameter}) excluding $\epsilon \rightarrow 0$ here and below in this proof.\\
\indent 
If we require (\ref{condition_for_intersection_point}) (and we consider only $Y \geq 0$), then there exist at least one and at the most three singular points (see Figure \ref{fig:curve_IS_LM_2} and try to shift the curve IS or LM downwards or upwards). Let denote $J$ as the Jacobi's matrix of the system (\ref{dynamic_model_LM_parameter}) and $\lambda_{1,2}$ as the eigenvalues of this Jacobi's matrix. It holds 
\begin{equation}
\lambda_{1,2} = \frac{1}{2} \left[ \epsilon \alpha (I_Y - S_Y) + \beta (L_R-M_R) \pm \sqrt{\left[ \epsilon \alpha (I_Y - S_Y) + \beta (L_R - M_R)  \right]^2 - 4 detJ} \right] \nonumber
\end{equation}
where  $det J = \epsilon \alpha \beta \left[ (I_Y - S_Y)(L_R-M_R) - (I_R - S_R) (L_Y-M_Y) \right]$.\\
\indent
Now, we consider three singular points. The singular point $E_1$ or $E_3$ lies in the part of arc $A_1$ or $A_3$ (see the Figure \ref{fig:curve_IS_LM_2} and \ref{fig:curve_LM}). So, we have to prove that these two singular points are stable points. Then, this implies that arcs $A_1$ and $A_3$ are stable arcs. The real part of eigenvalues $Re(\lambda_{1,2}) < 0$ in the point $E_1$ or $E_3$ because it holds $I_Y < S_Y$ according to the condition (\ref{economic_I_S}), $L_R < M_R$ according to the condition (\ref{economic_L}) and (\ref{economic_M_R}) and it holds $det J>0$ according to conditions (\ref{economic_I}), (\ref{economic_S}), (\ref{economic_L}), (\ref{economic_M_Y}), (\ref{economic_M_R}) and (\ref{economic_I_S}) (the point $E_1$ or $E_3$ lies in the area of $R \in [0,P] \cup [Q, \infty)$). From this follows that $E_1$ and $E_3$ are attractors. \\
The singular point $E_2$ lies in the part of arc $A_2$ (see the Figure \ref{fig:curve_IS_LM_2} and \ref{fig:curve_LM}). So, we have to prove that this singular point are unstable point, more precisely unstable saddle point. Then, this implies that arcs $A_2$ is unstable arc. We can see that the slope of the curve IS is smaller than the slope of the curve LM in the point $E_2$ (see Figure \ref{fig:curve_IS_LM_2}). Using the Implicit Function Theorem for the neighbourhood of the point $E_2$, it holds $-\frac{I_Y - S_Y}{I_R - S_R}<-\frac{L_Y-M_Y}{L_R-M_R}$. This and the conditions (\ref{economic_I}), (\ref{economic_S}), (\ref{economic_M_Y}), (\ref{economic_I_S}), (\ref{unusual_economic_L_R}) and (\ref{unusual_economic_M_R}) (the point $E_2$ lies in the area of $R \in (P,Q)$) imply that determinant of Jacobi's matrix of the system (\ref{dynamic_model_LM_parameter}) in the point $E_2$ is negative. From this follows that $E_2$ is unstable saddle point. \\
\indent
Now, we consider two singular points. First singular point is qualitatively the same as the point $E_1$ or $E_3$ from previous situation with three singular point. This singular point lies in the area of the arc $A_1$ or $A_3$. So, these arcs are stable. \\
The second point describe levels of $Y$ and $R$ where the curve IS and the curve LM has the same slope. Using the Implicit Function Theorem for the neighbourhood of the point $E_2$, it holds $-\frac{I_Y - S_Y}{I_R - S_R}=-\frac{L_Y-M_Y}{L_R-M_R}$. This implies zero determinant of Jacobi's matrix of the system (\ref{dynamic_model_LM_parameter}) in this point. This leads to at least one zero eigenvalue. This singular point is surely unstable and lies in the area of the arc $A_2$. So, from this follows that $A_2$ is unstable also in this situation of two singular points. \\
\indent
If there originates only one singular point, this singular point will be qualitatively the same as the point $E_1$ or $E_3$ from situation with three singular point. This singular point lies in the area of the arc $A_1$ or $A_3$. So, these arcs are stable.
\end{proof}

\begin{thm} (On existence of relaxation oscillations on money market) \\
Let us consider the new IS-LM model with very slow changes of variable $Y$ in time (\ref{dynamic_model_LM_parameter}) (including $\epsilon \rightarrow 0$) with economic properties (\ref{economic_I}), (\ref{economic_S}), (\ref{economic_I_S}) and three phases money demand and money supply. Then there exist the counterclockwise relaxation oscillations.
\end{thm}

\begin{figure}[ht]
  \centering
  \includegraphics[height=5.9cm]{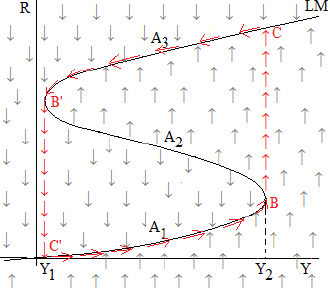}
  \caption{Relaxation oscillations in the model (\ref{dynamic_model_LM_parameter})}
  \label{fig:relaxation_oscillations_LM}
\end{figure}

\begin{proof}
If we have new IS-LM model with very slow changes of variable $Y$ in time (\ref{dynamic_model_LM_parameter}) considering $\epsilon \rightarrow 0$, then we can consider the system (\ref{dynamic_model_LM_parameter_2}) instead the system (\ref{dynamic_model_LM_parameter}).\\
\indent
Three phases money demand and money supply and the economic property (\ref{economic_M_Y}) of the money demand and money supply function ensure the form of the curve LM using the Implicit Function Theorem, see Figure \ref{fig:curve_LM}.\\
\indent
The trajectories of the system (\ref{dynamic_model_LM_parameter_2}) are directed almost vertically downwards or upwards (parallel to axis $R$) considering $\frac{d Y}{d t} = 0$, see the Figure \ref{fig:relaxation_oscillations_LM}. Up or down direction of the trajectories is given by the sign of the function $\alpha [I(Y,R)-S(Y,R)]$ on the curve LM. So, the direction of these trajectories is dependent on the stability or unstability of the arcs $A_1$, $A_2$ and $A_3$, see the Proposition \ref{prp:stable_unstable_arcs_2}. The trajectories are attracted to the stable arcs and are drove away the unstable arc. The speed of trajectories are finite near the isocline LM and nearness of the curve LM the trajectories go along the curve LM. The speed of trajectories are infinite large elsewhere.\\
\indent
For $Y<Y_1$, $Y>Y_2$ there exists only one stable stationary state and for $Y_1<Y<Y_2$ there exist three stationary states: two stable and one unstable between them. Every points of the curve LM are stationary points. $Y_1, Y_2$ are bifurcations values. If the parameter  $Y$ is changing very slowly, then the coordinates of stationary points are changed as displayed in the Figure \ref{fig:relaxation_oscillations_LM}.\\
\indent 
Now, we construct the cycle which is one vibration of the relaxation oscillations. We are changing the parameter $Y$ from the level $Y_1$ to $Y_2$. If the moving point is on or near the stable arc $A_1$, the moving point will go along this stable arc $A_1$, then it will pass the unstable arc $A_2$ from point $B$ to $C$, see Figure \ref{fig:relaxation_oscillations_LM}. The velocity between the point $B$ and $C$ is infinite large. There originates some "jump". There is the similar situation if we are changing the parameter $Y$ from the level $Y_2$ to $Y_1$. If the moving point is on or near the stable arc $A_3$, the moving point will go along the stable arc $A_3$ and then the moving point is attracted from point $B'$ to $C'$, see Figure \ref{fig:relaxation_oscillations_LM}. The sign of the function $\alpha[I(Y,R)-S(Y,R)]$ may change along the curve LM in such a way, that these trajectories form some path resembled limit cycle called relaxation oscillation. This oscillation contains the trajectories described by stable arcs $A_1$ and $A_3$ with finite velocity and the trajectories described by vertical segments (between points $B$ and $C$ and also between points $B'$ and $C'$) with infinite velocity (looking like "jump"), see the Figure \ref{fig:relaxation_oscillations_LM}. These trajectories form counterclockwise cycle because of the form of the curve IS and stability or unstability of the relevant arcs. Thus, we have just constructed the counterclockwise cycle which is one vibration of the the counterclockwise relaxation oscillations.
\end{proof}

We can display this oscillation using the isoclines IS and LM with shifting curve IS downwards and upwards, see the Figure \ref{fig:ralaxation_oscillations_IS-LM_2}.

\begin{figure}[ht]
  \centering
  \includegraphics[height=5cm]{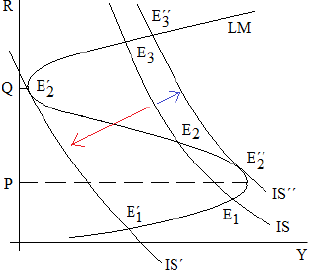}
  \caption{Shifts of the curve IS}
  \label{fig:ralaxation_oscillations_IS-LM_2}
\end{figure}

In this figure, we can see three singular points $E_1$, $E_2$ and $E_3$ (intersection points of IS and LM). $E_1$ and $E_3$ are stable singular points, $E_2$ is unstable saddle point. Then the curve IS shifts upwards or downwards according the changing of the parameter $Y$ until there are two singular points $E_2''$ and $E_3''$ (intersection points of IS'' and LM), or $E_1'$ and $E_2'$ (intersection points of IS' and LM). The point $E_1'$ and $E_3''$ is stable node or stable focus and the point $E_2'$ and $E_2''$ is unstable, see the proof of the Proposition \ref{prp:stable_unstable_arcs_2}.

\section*{Conclusion}
In this paper, we introduced new IS-LM model. This new model differs from the original model in elimination of its two main deficiencies. So, the new IS-LM model includes the modelling of inflation and the conjunction of exogenous and endogenous money supply in addition to the original model.\\
\indent
In this new IS-LM model, there can originate a special type of oscillations resembled limit cycle called relaxation oscillations. If we have new IS-LM with very slow changes of the interest rate in time in proportion to the aggregate income and with Kaldor's conditions, then there will exist relaxation oscillations on goods market. Likewise, if we have new IS-LM with very slow changes of the aggregate income in time in proportion to the interest rate and with three phases money demand and money supply, then there will exist relaxation oscillations on money market or financial assets market. We also described these oscillations by shifts of the curve LM or IS.

\section*{Discussion}
In these days, many experts and also the public more and more talk about unexpected fluctuation of different phenomenons in economics and about impact of these fluctuations on economics. In this paper, we show one point of view on this problems. We try to model some unexpected fluctuation of aggregate income on goods market and of interest rate on money or financial assets market using own new IS-LM model and theory of relaxation oscillations.\\
\indent
Relaxation oscillations on goods market cause quick change of the aggregate income which seems to be unexpected. Similarly, relaxation oscillations on money market (or financial assets market) cause quick change of the interest rate which seems to be unexpected. This quick "jumps" can ascribe unusual behaviour of economics.\\
\indent
The appropriate shifting of the curve IS or LM causes the relaxation oscillations. The fiscal policy can be demonstrated by the shifts of the curve IS downwards or upwards, vice versa the monetary policy can be demonstrated by the shifts of the curve LM downwards or upwards. This new IS-LM model with relaxation oscillations can estimate the impacts of the government intervention in this situation. If the interest rate is changed (monetary policy), then the curve LM will shift and the relaxation oscillations on goods market can originate (depending on the level of interest rate). Then, there will originate described quick jump of the aggregate income. If the aggregate income is changed (fiscal policy), then the curve IS will shift and the relaxation oscillations on money or financial assets market can originate (depending on the level of aggregate income). Then, there will originate described quick jump of the interest rate. These quick jumps of aggregate income on goods market, or of interest rate on money and financial assets market, can explain seemingly unexpected fluctuation of these quantities.\\
\indent
Today's, it seems to be necessary to model the ostensibly inexplicable phenomenons in economics.

\section*{Acknowledgements}
The research was supported, in part, by the Student Grant Competition of Silesian University in Opava, grant no. SGS/19/2010.

\end{document}